\documentclass[12pt]{amsart}
\usepackage{amssymb,enumerate}
\usepackage{graphicx}
\newcommand{\subfloat}{}

\newcommand{\X}{{\mathbf{R}^d}}

\newcommand{\Y}{{\mathbf{R}^d}}

\newcommand{\R}{{\mathbf R}}

\newcommand{\Leb}{{\mathcal{L}}}

\theoremstyle{plain}

\newtheorem{thm}{Theorem}[section]
\newtheorem{cor}[thm]{Corollary}
\newtheorem{lem}[thm]{Lemma}
\newtheorem{problem}[thm]{Problem}
\newtheorem{propn}[thm]{Proposition}

\begin{document}

\title [Capacity constrained optimal transport]{Insights into
capacity constrained optimal transport}
\date{\today}

\author{Jonathan Korman}
\email{jkorman@math.toronto.edu}

\author{Robert J. McCann}
\thanks{The second author is pleased to acknowledge the support of Natural Sciences and Engineering Research Council of Canada Grants 217006-08. \copyright 2012 by the authors.}
\email{mccann@math.toronto.edu}
\address{Department of Mathematics, University of Toronto, Toronto Ontario M5S 2E4 Canada}



\begin{abstract}
A variant of the classical optimal transportation problem is:
among all joint measures with fixed marginals and which are dominated
by a given density, find the optimal one. Existence and uniqueness
of solutions to this variant were established in~\cite{KM11}.
In the present manuscript, we expose an unexpected symmetry leading to the first explicit
examples in two and more dimensions.  These are inspired in part by simulations in one
dimension which display singularities and topology and in part by two further developments:
the identification of all extreme points in the feasible set,  and a new approach
to uniqueness based on constructing feasible perturbations.
\end{abstract}

\maketitle

\setcounter{secnumdepth}{2}

\section{Introduction}


Given fixed distributions of supply and demand,
the optimal transportation problem of Monge \cite{Monge81} and
Kantorovich \cite{Kantorovich42} involves pairing supply with demand
so as to minimize the average transportation cost $c(x,y)$ between each
supplier $x$ and the demander $y$ with whom $x$ is paired.
For continuous distributions, 
this question forms an (the?)
archetypal example of an infinite-dimensional linear program. 
Its relevance to the physics of fluids has been recognized since the work of 
Brenier \cite{B87} and Cullen and Purser \cite{CP89},
while some of its applications to geometry,  dynamics,  partial differential equations,
economics and statistics 
are described in \cite{MG} \cite{RachevRueschendorf} 
\cite{Villani2} and the references there.
It is desirable to introduce congestion effects into this model, as can be attempted
in various ways \cite{CJS08};  one of the crudest is simply to bound the number of 
suppliers at $x$ who can be paired with demanders at $y$,  for each $x$ and $y$.  Despite
its appeal, for continuous distributions of supply and demand, this variant seems not
to have been studied until \cite{KM11}.  

As in all linear programs,  if the problem has solutions,  at least one of them will be 
an extreme point of the feasible set.  A remaining challenge in the Monge-Kantorovich 
transportation problem  
is to arrive at a characterization of the extreme points
which yields useful information about the geometry and topology of its solutions \cite{AhmadKimMcCann11}.
Somewhat surprisingly, such a characterization is much more accessible in our capacity constrained variant;
as shown below, it can basically be reduced to a `bang-bang' 
(all or nothing) principle. As a corollary, this characterization implies the uniqueness of solutions 
first established in \cite{KM11}.  Moreover,  it combines with elementary but obscure symmetries
to yield the first explicitly soluble examples in more than one dimension,  and with numerical and theoretical
considerations to give insights into the geometry and topology aspects of
basic examples which --- even in one-dimension --- still defy explicit solution.  

The problem in question is formulated precisely as follows:
Given densities $0 \leq f, g \in L^1_{} (\R^d)$ with same total mass
$\int f = \int g$, let  $\Gamma(f, g)$ denote the set of joint densities
$0 \leq h \in L^1_{}( \X \times \Y)$ which have $f$ and $g$ as their marginals,
meaning
\begin{eqnarray*}
f(x) = &\displaystyle \int_{\Y} h(x,\tilde y)d\tilde y & \\
&\displaystyle \int_{\X} h(\tilde x,y)d\tilde x &= g(y)
\end{eqnarray*}
for Lebesgue almost all $x, y \in \R^d$.
A bounded function $c(x,y)$ represents the cost per unit mass for transporting material
from $x\in\X$ to $y\in\Y$. 
The (total) transportation cost of $h$ is denoted $c[h]$, defined by
\begin{eqnarray}\label{transportation_cost}
c[h]:=\int_{\X \times \Y} c(x,y)h(x,y)dx dy,
\end{eqnarray}
is proportional to the expected value of $c$ with respect to $h$.

Given $0 \leq \bar h\in L^1(\X \times \Y)$, 
we let $\Gamma(f, g)^{\bar h}$ denote the set of all $h\in\Gamma(f, g)$ dominated by $\bar h$, that is $h \leq \bar h$ almost everywhere.
The optimization problem we are concerned with --- {\it optimal transport with capacity constraints}---is to minimize the transportation cost (\ref{transportation_cost})
among joint densities $h$ in $\Gamma(f, g)^{\bar h}$, to obtain the optimal cost
\begin{eqnarray}\label{eqn: constraind_optimal_cost}
\underset{h\in\Gamma(f, g)^{\bar h}}{\min} c[h]
\end{eqnarray}
under the capacity constraint $\bar h$.

\smallskip

Notice this problem involves a linear minimization on a convex set
$\Gamma = \Gamma(f, g)^{\bar h}$, and therefore takes the form
of an infinite-dimensional linear program.
When $\Gamma$ is non-empty,
it is not hard to show that the minimum is attained, 
and that at least one of the minimizers $h$ is an extreme point of $\Gamma$,
meaning $h$ is not 
the midpoint of any segment in $\Gamma$.
Since the possible facet structure of $\Gamma$ is not obvious,
it is harder to determine whether or not this minimizer is unique.
A sufficient condition for uniqueness was discovered
in \cite{KM11}, and is recalled below.  It is even harder to envision what the solutions
will look like.  Perhaps the simplest example involves pairing Lebesgue
measure on the unit interval with itself so as to minimize
the quadratic transportation cost $c(x,y)=|x-y|^2$
(well-known to be equivalent to $c(x,y)=-x \cdot y$, as in the
 unconstrained problem \cite{B87}).
As the capacity
constraint $\bar h$ is varied over different constant values,
the numerical solutions below display an unexpected variety of strange
topological features and analytic singularities begging to be understood.
Even though supply equals demand in these examples,  and the constraints permit
at least some of the demand to be supplied locally at zero cost,
the islands of blue in these diagrams show that global optimality
may require there to be regions where none of the demand is supplied locally.

In this paper we establish symmetries which explain at least some of the observed
structures (Lemma \ref{lem:symmetries}).
These also lead to the first explicit examples of optimizers
in higher dimensions (Proposition \ref{P:examples}).  We precede this with a simple
description of the extreme points of $\Gamma$,  based on a perturbation argument.
The same idea is also used to substantially simplify the uniqueness argument
from~\cite{KM11}. The original argument relied on understanding the infinitesimal behaviour
of an optimizer near its Lebesgue points in order to argue that any optimizer is
geometrically extreme---a property which characterizes extreme points of
$\Gamma(f, g)^{\bar h}$ (see proposition~\ref{P:all or nothing}).
The new proof begins from an `all or nothing' characterization
of extreme points
and uses perturbations to argue directly --- without asymptotics or blow ups ---
that every optimizer is extreme. Uniqueness follows easily.\\

\noindent {\bf Acknowledgements}. We would like to thank Brian Wetton for sharing
the figures and the MATLAB code that generated them with us; these simulations
inspired Lemma~\ref{lem:symmetries}.

\section{Assumptions}\label{section:notation}

We make the following assumptions throughout (see~\cite{KM11} for more details).

\subsection{Assumptions on the cost}\label{assumptions:cost}
\begin{itemize}
\item[(C1)] $c(x,y)$ is bounded,
\item[(C2)] there is a Lebesgue negligible closed set $Z \subset \X \times \Y$ such that $c(x,y)\in C^2(\X \times \Y \setminus Z)$ and,
\item[(C3)] $c(x,y)$ is non-degenerate: $\det [D^2_{x^iy^j}c(x,y)]\neq 0$ for all $(x,y)\in \X \times \Y \setminus Z$.
\end{itemize}

\subsection{Assumption on the capacity constraint}\label{assumptions:capacity} $\bar h$ is non-negative and
Lebesgue integrable:  $0 \le \bar h \in L^1(\X \times \Y)$.

Given marginal densities $0\leq f, g\in L^1_{}(\X)$ with same total mass,
to avoid talking about the trivial case, we will always assume that a
feasible solution exists: $\Gamma=\Gamma(f, g)^{\bar h}\neq \emptyset$.

\section{Uniqueness: every optimizer is extreme}

Our arguments are based crucially on a preparatory lemma from real analysis.

\begin{lem}[Marginal-preserving volume exchange]\label{L:exchange}
Fix $\Delta = (1,\ldots,1) \in \R^d$.  If a Lebesgue set $U \subset \R^d \times \R^d$
is non-negligible,  then for all $\delta>0$ sufficiently small there is a subset
$V \subset U$ of positive volume such that
$(x+\delta\Delta,y),(x,y+\delta\Delta)$ and $(x+\delta\Delta,y+\delta\Delta)$ all
belong to $U$ whenever $(x,y) \in V$.  Moreover, $V$ may be taken to lie in the 
interior of a
coordinate hypercube of side-length $\delta$. The vertex of this hypercube at
which $\Delta$ is an outward normal may be chosen to lie at any Lebesgue point $z_0$
where $U$ has full density.
($V$ may be chosen to have Lebesgue density $1/2^{2d}$ at $z_0$.)
\end{lem}

\begin{proof}
Let $I=[0,1]$ denote the unit interval and $I^d$ the hypercube, so
that $H_{ij} = (2i-1) I^d \times (2j-1)I^d$ for $i,j \in \{0,1\}$
define four hypercubes with disjoint interiors in $2d$ dimensions.
Let $z_0$ be a Lebesgue point where $U$ has full density;
we may suppose $z_0$ is the origin without loss of generality.  Letting $\delta^{-1}U$
denote the dilation of $U$ around $z_0$ by factor $\delta^{-1}$, we see
the fraction of $H_{ij}$ outside of $\delta^{-1}U$ tends to zero as $\delta \to 0$.
For $\delta$ sufficiently small, we may assume all four of these fractions to
be strictly less than $1/4$.  Let $V \subset H_{00}$ be the set of
$(x,y) \in H_{00} \cap \delta^{-1} U$
for which $(x+\Delta,y)$, $(x,y+\Delta)$ and $(x+\Delta,y+\Delta)$ also belong to $\delta^{-1}U$.
If $(x,y) \in H_{00} \setminus V$,  it is because at least one of the four points above
does not belong to $\delta^{-1}U$. Thus
$H_{00} \setminus V = J_{00} \cup J_{01} \cup J_{10} \cup J_{11}$
where each of the four sets
$J_{ij} + (i\Delta,j\Delta) := H_{ij} \setminus \delta^{-1} U$
has volume strictly less than $1/4$.  Thus $\Leb^{2d}[H_{00} \setminus V]<1$,
implying $V$ is a set of positive measure.  Discarding from $V$ any points on the boundary of $H_{00}$
and contracting by a factor $\delta$ yields the lemma.
(The parenthetical remark is obtained by noting $1/4$ is arbitrary in the argument above;
 taking $\delta$ smaller forces the volume of $H_{00} \setminus V$ to be as small as we
 please.  Thus $V$ fills a larger and larger fraction of the hypercube $H_{00}$
 near its vertex, where both $H_{00}$ and hence $V$ have Lebesgue density $1/2^{2d}$.)
\end{proof}

This allows us to give a much nicer characterization of the extreme points
of $\Gamma(f,g)^{\bar h}$ than any 
available for the
unconstrained problem ($\bar h=+\infty$) \cite{AhmadKimMcCann11}.

\begin{propn}[All or nothing characterization of extreme points]
\label{P:all or nothing}
Let $\Gamma=\Gamma(f,g)^{\bar h}$ denote the set of joint densities bounded by
$\bar h \in L^1(\R^n \times \R^n)$ and with marginals $f,g \in L^1(\R^n)$.
A density $h \in \Gamma$ is an extreme point of $\Gamma$
if and only if $h = 1_W \bar h$ for some Lebesgue measurable set
$W \subset \R^d \times \R^d$.
\end{propn}

\begin{proof}
Recall $h \in \Gamma$ implies $0 \le h \le \bar h$.
If $h$ is extremal, we claim these inequalities cannot both be strict
on any subset $U\subset \R^{2d}$ of positive volume.  To show the contrapositive,
suppose such a $U$ existed.  Then for some $\epsilon>0$
the set $U_\epsilon = \{z \in \R^{2d} \mid \epsilon< h< \bar h-\epsilon\}$
would also have positive volume.  Lemma \ref{L:exchange} provides $\delta>0$
and $V \subset U_\epsilon$ of positive measure such that all four points
$(x \pm \frac{\delta}{2} \Delta, y\pm \frac{\delta}{2} \Delta)$ and
$(x \pm \frac{\delta}{2} \Delta, y\mp \frac{\delta}{2} \Delta)$ lie in $U_\epsilon$
whenever $(x -\frac{\delta}{2} \Delta, y-\frac{\delta}{2} \Delta) \in V$.  Setting
$I=[0,1]$ and using $i,j \in \{0,1\}$ to define four coordinate
hypercubes $H_{ij}=(2i-1)I^d \times (2j-1)I^d$,
after translation we may also assume $V \subset H_{00} \setminus \partial H_{00}$.
Then
\begin{equation}\label{neutral perturbation}
\tilde h(x,y) := \left\{
\begin{array}{ll}
+1 & {\rm if}\ (x,y) \in V\ {\rm or}\ (x-\delta\Delta,y-\delta\Delta) \in V\cr
-1 & {\rm if}\ (x-\delta\Delta,y) \in V\ {\rm or}\ (x,y-\delta\Delta) \in V \cr
 0 & {\rm otherwise}
\end{array}\right.
\end{equation}
is well-defined.  Notice that $\tilde h$ is constructed using symmetries which ensure
its integrals with respect to $x$ and with respect to $y$ both vanish, the other variable
being held fixed.  In other words,  the marginals of $\tilde h$ vanish.
Also, $\tilde h$ is supported in $U_\epsilon$, where we have room to add or subtract
$\epsilon$ from $h \in (\epsilon,\bar h-\epsilon )$.
Thus $h_\pm := h \pm \epsilon \tilde h$ both belong to $\Gamma$; they are distinct
since $V$ has positive volume.  Expressing
$h= \frac{1}{2} (h_+ + h_-)$ as a convex combination of $h_\pm$
shows $h$ is not an extreme point of $\Gamma$.

Conversely, we claim any $h=1_W \bar h$ with $W \subset \R^{2d}$ Lebesgue is extreme.
To see this, suppose $h=1_W \bar h$ could be decomposed as a convex combination
$h=\frac{1}{2}(h_+ + h_-)$ of $h_\pm \in \Gamma$.
Since $h_\pm$ are both non-negative,  they must both vanish where $h$ does;
thus $h_\pm =0$ outside of $W$.  Since both $h_\pm \le \bar h$,  they must
both coincide with $\bar h$ where $h$ does;  thus $h_\pm = \bar h$ in $W$.
This shows $h_+ =h_-$, establishes extremality of $h=1_W \bar h$,
and completes the proof of the proposition.
\end{proof}

More importantly, it allows us to construct a perturbative argument
for uniqueness,  much simpler than the original proof of \cite{KM11}.

\begin{thm}[Every optimizer is extreme]\label{T:optimizers are extreme}
Let the cost $c(x,y)$ satisfy conditions $(C1)-(C3)$, 
fix $0 \leq \bar h\in L^1 (\X \times \Y)$ and take
$0 \leq f, g \in L^1_{}(\X)$ such that
$\Gamma:=\Gamma(f,g)^{\bar h}\neq \emptyset$.
If $h\in\Gamma$ is optimal, i.e. $ h \in \mbox{\rm argmin}_{k\in\Gamma} c[k]$, then $h$ is
an extreme point of $\Gamma$.
\end{thm}

\begin{proof}
Suppose $h \in \Gamma$ is not an extreme point of $\Gamma$.
We shall establish the theorem by constructing a perturbation of $h$
which decreases the cost $c[h]$.  Proposition \ref{P:all or nothing}
asserts $h \ne 1_W \bar h$,  meaning the set $U$ of Lebesgue points $z$
for $h$ and $\bar h$ at which $0<h<\bar h$ has positive volume.
Similarly, for sufficiently small $\epsilon>0$ the set
$U_\epsilon = \{ z\in U \setminus Z \mid \epsilon < h < \bar h -\epsilon\}$
also has positive volume, where $Z$ is the negligible closed set on which hypotheses
(C2) $c \in C^2$ and (C3) $\det [D^2_{x^i y^j}]\ne 0$ may fail.
Let $z_0=(x_0,y_0)$ be a point where
$U_\epsilon$ has full Lebesgue density;  we may assume $z_0$ to be the origin
without loss of generality.  After a linear transformation 
of the variable $y$ (as in~\cite{MPW} or \S 5 of \cite{KM11}),
we can also assume $D^2_{x^i y^j} c (z_0)= -\delta_{ij}$
without losing generality.
Set $I=[0,1]$ and
$H_{ij} = (2i-j)I^d \times (2j-1) I^d$ for $i,j \in \{0,1\}$.
Applying Lemma \ref{L:exchange} in the new coordinates 
yields $\delta>0$ and a set $V \subset  H_{00}$ of positive measure
such that $(x+i\delta \Delta,y+j\delta \Delta) \in U_\epsilon \cap H_{ij}$
for $i,j \in \{0,1\}$.
The perturbation $\tilde h_\delta =\tilde h$ of \eqref{neutral perturbation} is again well-defined,
and its marginals vanish.  Moreover,
$h_{\delta} = h + \epsilon \tilde h_\delta \in \Gamma$
is a feasible competitor since
$|\tilde h_\delta| \le 1_{U_\epsilon}$ and $h \in (\epsilon,\bar h-\epsilon )$ on $U_\epsilon$.
The change in cost produced by this perturbation is
\begin{eqnarray*}
&& c[h_{\delta}] - c[h] \cr
&=& \epsilon \int c(x,y) \tilde h_\delta (x,y) d^d x d^dy \cr
&=& \epsilon \int_V [c(x,y) + c(x+\delta \Delta,y+\delta \Delta) - c(x + \delta \Delta,y)
- c(x,y+ \delta \Delta) ] d^dx d^dy \cr
&=& \epsilon \delta^2 \int_V [  \int_0^1 \int_0^1
\sum_{i,j=1}^n D^2_{x^iy^j} c (x+s\delta \Delta,y+t\delta \Delta) ds dt ]d^dx d^dy.
\end{eqnarray*}
In this formula, the arguments $z$ of the continuous mixed partials $D^2_{x^i y^j} c$
all lie within distance $\delta \sqrt{2d}$ of a point $z_0$ at which
$\sum_{i,j} D^2_{x^iy^j}c(z_0) = -n$. Thus for $\delta$ small enough, the
perturbed cost $c[h_{\delta} - h]<0$ is 
negative, precluding optimality of $h$.
The contrapositive implies the only optimizers $h$ of $c$ are extreme points of $\Gamma$.
\end{proof}

\begin{cor}[Uniqueness of Optimizer]\label{cor:uniqueness}
Under the same hypotheses,  the minimum in Theorem \ref{T:optimizers are extreme}
is uniquely attained.
\end{cor}

\begin{proof}
If $h_0$ and $h_1$ both minimize $c[h]$ on $\Gamma$,
then so does $h_{1/2} = \frac{1}{2}(h_0+h_1)$ since
$c[\,\cdot\,]$ is linear and $\Gamma$ is convex.
Theorem \ref{T:optimizers are extreme} then asserts extremality of $h_{1/2}$ in $\Gamma$,
so $h_0=h_1$.  This is the desired uniqueness.
\end{proof}

\section{Simulations and symmetries}

In case $\bar h(x,y)= const$,  the problem has symmetries which limit the
possible solutions.  After introducing these symmetries,  we use them to
establish a new class of examples 
for which the optimal transport can be
displayed explicitly.  These include the two-by-two checkerboard
(Example 1.1 of~\cite{KM11} or Corollary \ref{cor:checker-board} below) as a particular case.\\

Figure~\ref{fig:hbar} shows a simulation of the optimal solutions with uniform marginals
for the distance squared cost on $I \times I$ with $\bar{h}=3$ and with
$\bar{h}=\frac{3}{2}$. Red represents the region $W$ where the $\bar{h}$ constraint is
saturated; in the complementary blue region, no transportation occurs.
These computer simulations were originally presented to us by Brian Wetton who remarked
on the symmetry manifested between the $\bar{h}=3$ case and the $\bar{h}=\frac{3}{2}$ case.
This symmetry is explained by the following lemma,  which applies to any pair of
H\"older conjugates $p$ and $q$.

\begin{figure}[h!]
\centering
\subfloat[$\bar{h}=3$]{
\label{fig:subfig:a}
\includegraphics[width=1.7in]{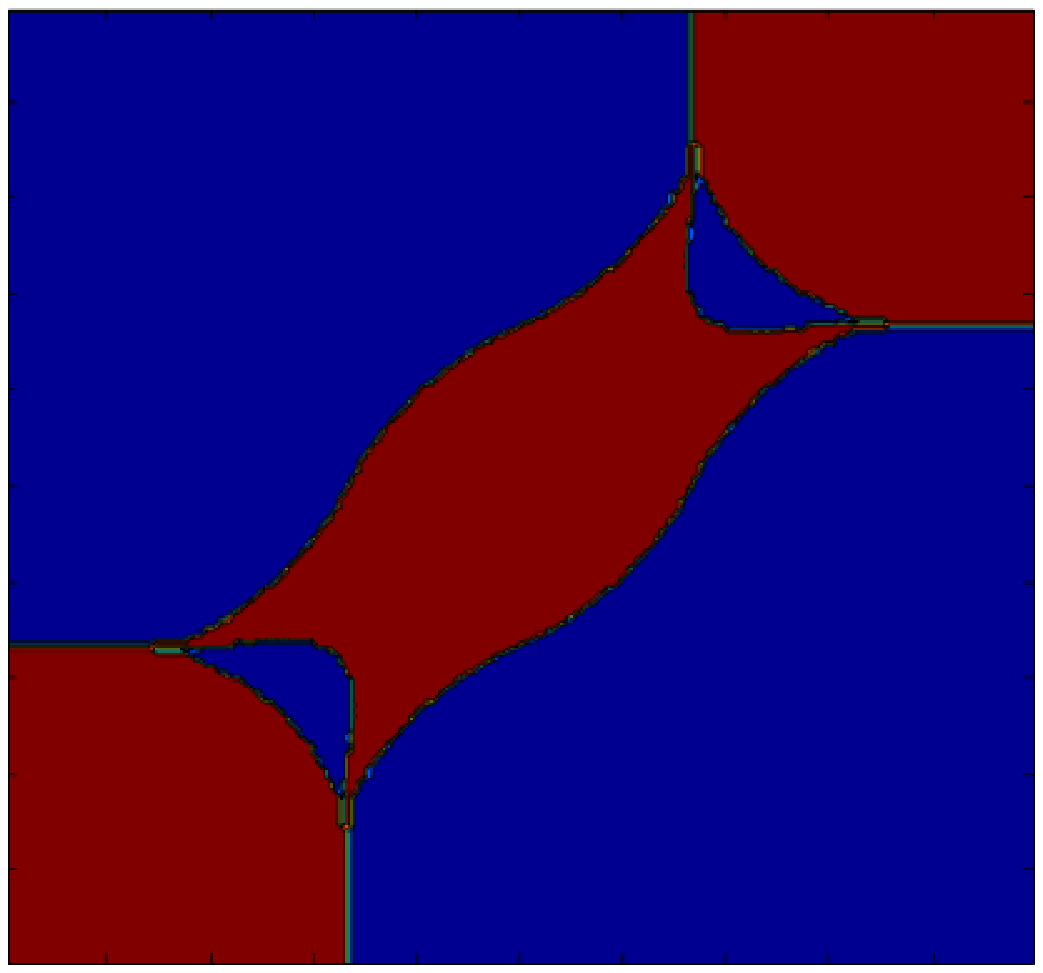}}
\hspace{0.1in}
\subfloat[$\bar{h}=\frac{3}{2}$]{
\label{fig:subfig:b}
\includegraphics[width=1.7in]{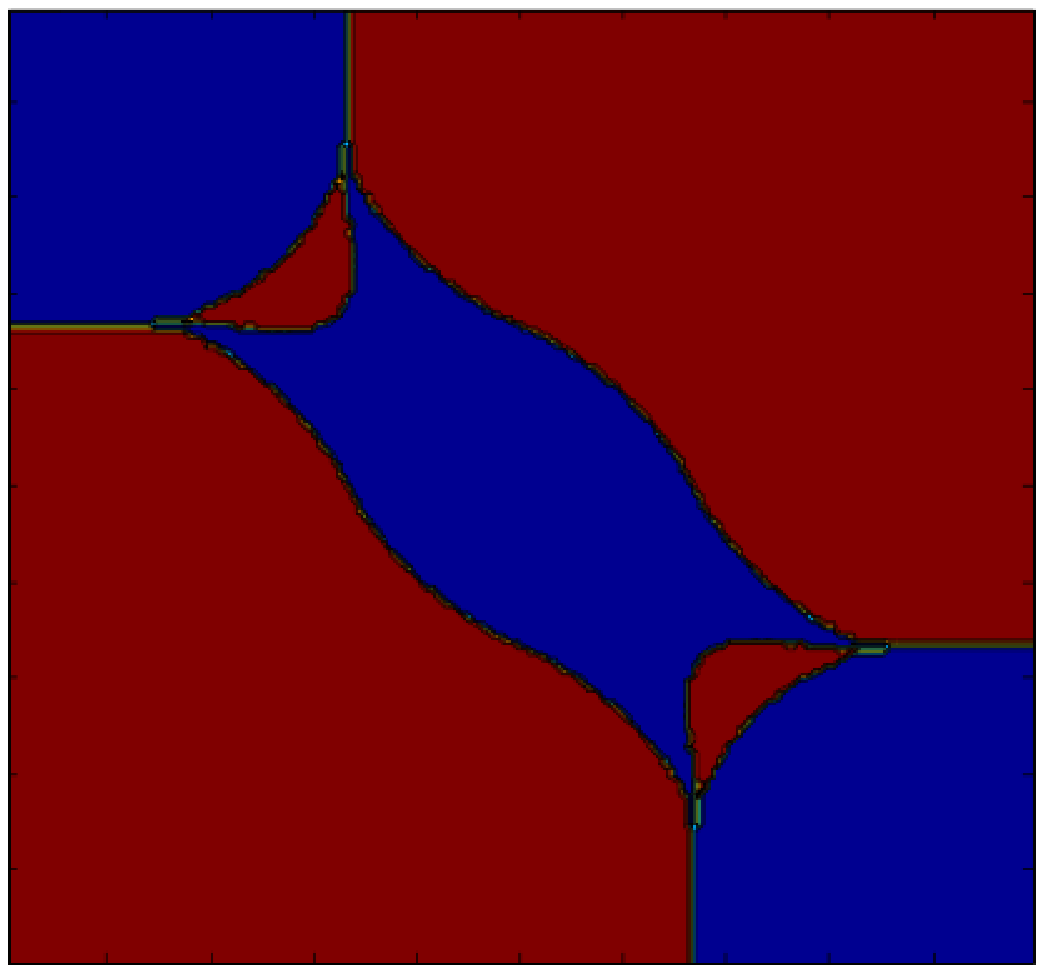}}
\caption{Red represents the saturation set $W$ given $\bar h$.
Note the symmetry between (a) $\bar h =3$ and (b) $\bar h= 3/2$.}
\label{fig:hbar}
\end{figure}

\begin{lem}[Symmetries]\label{lem:symmetries}
Let $\Omega,\Lambda \subset \R^d$ be bounded sets with unit volume.
Set $f=1_\Omega$ and $g=1_\Lambda$, 
where $1_\Omega$ denotes the indicator function of the set $\Omega$.
Let $\bar h=\bar h_p:=p  1_{\Omega \times \Lambda}$ have constant density
$p>1$ on the product $\Omega \times \Lambda$, and $c(x,y)=-x\cdot y$.  Given any
set $W \subset \Omega \times \Lambda$,  let $R(W)$ denote its image under the
reflection $R(x,y)=(x,-y)$,  and $\tilde W := (\Omega\times\Lambda) \setminus W$
its set theoretic complement.  If 
$p 1_W \in \Gamma^{\bar h_p}(1_\Omega,1_\Lambda)$ then 
$q 1_{\tilde W} \in \Gamma^{\bar h_q}(1_\Omega,1_\Lambda)$, where
$p^{-1} + q^{-1} =1$ are H\"older conjugates.
Moreover
$c[1_W]+b(\Omega)\cdot b(\Lambda)=-c[1_{\tilde W}]=c[1_{R(\tilde W)}]$,
where $b(\Omega) = \int_\Omega x$ is the center of mass of $\Omega$.
Thus 
$p 1_W$ minimizes $c$ on $\Gamma^{\bar
h_p}(1_\Omega,1_\Lambda)$ if and only if $q 1_{R(\tilde W)}$ minimizes $c$ on
$\Gamma^{q1_{R(\Omega \times \Lambda)}}(1_\Omega,1_{-\Lambda})$.
\end{lem}

\begin{proof}
For $W \subset \Omega \times \Lambda$ set
$W(x) = \{ y \in \R^d \mid (x,y) \in W\}$ and
$W^{-1}(y) = \{x \in \R^d \mid (x,y) \in W\}$.
Notice $p 1_W \in \Gamma^{\bar h_p}(1_\Omega,1_\Lambda)$ if and only if
$p|W(x)|=1$ and $p|W^{-1}(y)|=1$ for
a.e.~$(x,y)\in \Omega \times \Lambda$.
From $|W(x)| + |\tilde W(x)|=1$ we conclude
$|\tilde W(x)| = 1 - \frac{1}{p} = \frac{1}{q}$
for a.e. $x \in \Omega$,  and similarly
$|\tilde W^{-1}(y)|=\frac{1}{q}$.  Thus
$q 1_{\tilde W} \in \Gamma^{\bar h_q}(1_\Omega,1_\Lambda)$.


On the other hand,
\begin{eqnarray*}
c[1_W]+c[1_{\widetilde{W}}]
&=&\int_W c(x,y)+ \int_{\widetilde{W}} c(x,y)
\\ &=& \int_{\Omega \times \Lambda} -x\cdot y
\\ &=& - \int_\Omega x  \cdot \int_\Lambda y
\\ &=& - b(\Omega) \cdot b(\Lambda)
\end{eqnarray*}
and
\begin{eqnarray*}
c[1_{R(\widetilde{W})}]
&=& \int_{R(\widetilde{W})} c(x,y)
\\&=& \int_{\widetilde{W}} c(x,-y)
\\&=& - \int_{\widetilde{W}} c(x,y)
\\&=& - c[1_{\widetilde{W}}],
\end{eqnarray*}
which imply the remaining assertions.
\end{proof}

The next proposition shows these elementary symmetries yield a broad class of
examples in the self-dual case $p=2=q$.

\begin{propn}[Universal optimizer for a balanced set with self-dual constraint]
\label{P:examples}
Fix uniform densities $f=1_{\Omega}$ and $g=1_\Lambda$ on two
bounded Lebesgue sets $\Omega,\Lambda \subset \R^d$ of unit volume.
Let $\bar h=\bar h_2 := 2 \cdot 1_{\Omega \times \Lambda}$ have constant density $2$ on
$\Omega \times \Lambda$,
and fix $c(x,y)=-x \cdot y$. If $\Omega$ and $\Lambda$ are balanced,
meaning $\Omega=-\Omega$,
the minimizer $\bar h_2 1_W$ of $c$ on $\Gamma^{\bar h_2}(1_\Omega,1_\Lambda)$ satisfies
$W = R(\tilde W) = -R(\tilde W)$ (up to sets of measure zero), where $R(x,y) = (x,-y)$.
It follows that $W=\{ (x,y) \in \Omega \times \Lambda \mid x \cdot y >0\}$.
\end{propn}

\begin{proof}
Note 
that $\Gamma:=\Gamma(1_\Omega,1_\Lambda)^{\bar h}\neq \emptyset$ since it contains $1_{\Omega \times \Lambda}$.
Thus there exists $\bar h_2 1_W$ minimizing $c$ on $\Gamma$ 
as in \cite{KM11}.
Lemma \ref{lem:symmetries} ensures $\bar h_2 1_{R(\tilde W)}$
also minimizes $c$ on $\Gamma 
$, as do $\bar h_2 1_{-W}$
and hence $\bar h_2 1_{R({-\tilde W})} = \bar h_2 1_{-R(\tilde W)}$.
The uniqueness
established in Corollary~\ref{cor:uniqueness} therefore implies
$W = -W = R(\tilde W) = -R(\tilde W)$, up to sets of measure zero.
For each Lebesgue point
$z_0=(x_0,y_0) \in W$ of full density, this shows $W$ also contains
$(-x_0,-y_0)$
but not $R(z_0)=(x_0,-y_0)$ nor $-R(z_0) = (-x_0,y_0)$.
Therefore choose a Lebesgue point $z_0 =(x_0,y_0) \in W$
of full density with $x_0 \ne 0 \ne y_0$
and $x_0 \cdot y_0 \ne 0$, noting almost all points in $W$ have this form.
For $r>0$ sufficiently small, the ball $Z_1:=Z=B_r(z_0)$ will be disjoint from its reflections
$Z_3=-Z$, $Z_2=R(Z)$, and $Z_4=-R(Z)$, and moreover $(x,y) \in Z$ will imply $x \cdot y$
has the same sign as $x_0\cdot y_0$.  We claim $x_0 \cdot y_0>0$.
Otherwise $\tilde h = -1_{Z_1 \cup Z_3} + 1_{Z_2 \cup Z_4}$ would be
a feasible perturbation, and
$c(x,y) + c(-x,-y) > 0> c(x,-y) + c(-x,y)$
for all $(x,y) \in Z$ shows $\bar h_2 1_W + \tilde h$ lowers the
cost $c$ in  $\Gamma(f,g)^{\bar h_2}$. This contradicts the minimality
of $\bar h_2 1_W$.  Thus, up to sets of measure zero, $W$ must be contained in
$W' := \{ (x,y) \in \Omega \times \Lambda \mid x \cdot y>0\}$.
On the other hand,  the fact that $\Omega$ and $\Lambda$ are balanced makes it
easy to check feasibility of $\bar h_2 1_{W'}$.  Feasibility of $W$
then shows the containment $W \subset W'$ cannot be strict,  apart from
a set of measure zero,
so $W= W'$ as desired.
\end{proof}

As an immediate corollary,  we recover Example 1.1 of \cite{KM11},
displayed in Figure 2.  Note this analytical example (like the numerical
ones preceding) dispels a number of natural conjectures about the optimizing
set by demonstrating that its topology need not be simple and its
boundary need not be smooth.
\begin{figure}[h!]
\label{fig:checkerboard}
\centering{
\includegraphics[width=2.5in]{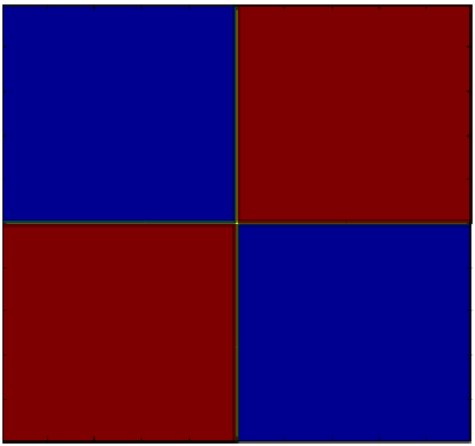}}
\caption{The two-by-two checkerboard solves $\bar h=2$.}
\end{figure}
Symmetry and self-duality gives a much more satisfactory explanation
for its singular nature than the original argument,  which was based
on guessing a solution to the linear program dual to \eqref{eqn: constraind_optimal_cost}.


\begin{cor}[The $2 \times 2$ checkerboard revisited]
\label{cor:checker-board}
Taking $\Omega =  \Lambda = [-\frac{1}{2},\frac{1}{2}]$, the preceding proposition
shows the minimizer $\bar h_2 1_W$ of $c[\,\cdot\,]$ on
$\Gamma^{\bar h_2}(1_\Omega,1_\Omega)$
to be given by $W=[-\frac{1}{2},0]^2 \cup [0,\frac{1}{2}]^2$.
\end{cor}

\section{Afterword}

When transport capacity between $x$ and $y \in \R^d$ is constrained by a density $\bar h \in L^1(\R^d \times \R^d)$,
the `all or nothing' (a.k.a.\ {\em bang-bang}) characterization of extremal plans 
$h \in \Gamma^{\bar h}(f,g)$ makes optimal transport between $f$ onto $g$ 
appear easier to analyze than the unconstrained problem $\bar h =+\infty$ \cite{AhmadKimMcCann11}.
Nevertheless,  the simple examples with capacity constraints solved above using numerical or theoretical
methods display an unexpectedly rich range of phenomena and raise new questions of their own.   
Although not discussed here, the linear program 
dual to the capacity constrained problem \cite{KM11} turns out to be more complicated to solve than 
that of the unconstrained problem $\bar h=+\infty$ \cite{RachevRueschendorf}  \cite{Villani2}.  We hope to 
analyze this difficulty in future work.




\end{document}